\documentclass[a4paper, dvipdfmx,12pt]{amsart}

\usepackage{amsmath,amssymb}
\usepackage{amsthm}

\usepackage{graphicx}
\usepackage{bm}
\usepackage{tikz-cd}
\usepackage{mathtools}
\usepackage{comment}
\usepackage{bigdelim, multirow}
\usepackage{latexsym}
\usepackage{stmaryrd}
\usepackage{braket}
\usepackage{color}
\usepackage{setspace}
\usepackage{comment}
\usepackage{url}
\newcommand{\Z}{\mathbb{Z}}

\newcommand{\R}{\mathbb{R}}

\newcommand{\calB}{\mathcal{B}}

\newcommand{\calI}{\mathcal{I}}

\newcommand\vphi{\varphi}

\newcommand\Ann{\mathop{\mathrm{Ann}}\nolimits}

\newcommand\Ker{\mathop{\mathrm{Ker}}\nolimits}

\newcommand\rank{\mathop{\mathrm{rank}}\nolimits}
\def\fr<#1/#2>{\frac{#1} {#2}}

\newcommand{\aaa}{{\bm a}}
\newcommand{\xx}{{\bm x}}

\newcommand{\SLP}{\mathrm{SLP}}
\newcommand{\HRR}{\mathrm{HRR}}

\newtheorem{thm}{Dont use this}[section]
\newtheorem{theorem}[thm]{Theorem}

\newtheorem{corollary}[thm]{Corollary}
\newtheorem{lemma}[thm]{Lemma}

\newtheorem{question}[thm]{Question}

\newtheorem*{definition*}{Definition}
\newtheorem*{theorem*}{Theorem}
\newtheorem*{proposition*}{Proposition}
\newtheorem*{corollary*}{Corollary}
\newtheorem*{lemma*}{Lemma}
\newtheorem*{problem*}{Problem}
\newtheorem*{question*}{Question}
\newtheorem*{conjecture*}{Conjecture}

\newtheorem{definition and lemma}[thm]{Definition $\&$ Lemma}
\newtheorem{definition and proposition}[thm]{Definition $\&$ Proposition}
\newtheorem{definition and corollary}[thm]{Definition $\&$ Corollary}
\newtheorem*{remark*}{Remark}
\theoremstyle{definition}
\newtheorem{definition}[thm]{Definition}
\newtheorem{example}[thm]{Example}

\newtheorem{remark}[thm]{Remark}

\theoremstyle{definition}

\theoremstyle{plain}
\title[Strictness of the log-concavity of matroids]{%
Strictness of the log-concavity of\\ generating polynomials of matroids}
\author[S. Murai]{Satoshi Murai}
\address[Satoshi Murai]{Department of Mathematics Faculty of Education Waseda University, 1-6-1 Nishi-Waseda, Shinjuku, Tokyo 169-8050, Japan}
\email{s-murai@waseda.jp}
\author[T. Nagaoka]{Takahiro Nagaoka}
\address[Takahiro Nagaoka]{Department of Mathematics, Graduate School of Science, 	Kyoto University, Kyoto, 606-8522, Japan}
\email{tnagaoka@math.kyoto-u.ac.jp}
\author[A. Yazawa]{Akiko Yazawa}
\address[Akiko Yazawa]{Department of Science and Technology,
	Graduate School of Medicine, Science and Technology,
	Shinshu University,
	Matsumoto, Nagano, 390-8621, Japan}
\email{yazawa@math.shinshu-u.ac.jp}

\date{}

\begin{document}
\pagestyle{plain}
\begin{abstract}
Recently, it was proved by
Anari--Oveis Gharan--Vinzant,
Anari--Liu--Oveis Gharan--Vinzant and Br\"{a}nd\'{e}n--Huh
that, for any matroid $M$, its basis generating polynomial and its independent set generating
polynomial are log-concave on the positive orthant. Using these, they obtain some combinatorial inequalities on matroids including a solution of strong Mason's conjecture.
In this paper, we study the strictness of the log-concavity of these polynomials and determine when equality holds in these combinatorial inequalities.
We also consider a generalization of our result to morphisms of matroids.
\end{abstract} 
\maketitle
\setcounter{tocdepth}{1}


\section{Introduction}\label{sec:Intro}

Given a matroid $M$ on $[n]=\{1,2,\dots,n\}$ of rank $r$,
one can associate two important polynomials
called the basis generating polynomial
and the independent set generating polynomial.
The \textit{basis generating polynomial} of $M$ is the polynomial
$$f_M= \sum_{B \in \mathcal B(M)} \left( \prod_{i \in B} x_i \right) \in \mathbb Z[x_1,\dots,x_n],$$
where $\mathcal B(M)$ is the set of bases of $M$.
The \textit{independent set generating polynomial} of $M$ is the polynomial
$$P_M=\sum_{I \in \mathcal I(M)} \left( \prod_{i \in I} x_i \right) x_0^{n-|I|} \in \mathbb Z[x_0,x_1,\dots,x_n],$$
where $\mathcal I(M)$ is the set of independent sets of $M$
and where $|X|$ denotes the cardinality of a finite set $X$.
It is also useful to consider the polynomial $\overline P_M:=(\frac {\partial} {\partial x_0})^{n-r} P_M$, which we call the \textit{reduced independent set generating polynomial} of $M$.

These polynomials catch interest of many researchers recently and have been actively studied from combinatorial and algebraic point of view. See e.g.\ \cite{AOV,ALOVI,ALOV,BH1,BH2,COSW,EH,MN,NY,Ya}.
Let $H_f=(\frac \partial {\partial x_i} \frac \partial {\partial x_j} f )$ be the Hessian matrix of a polynomial $f$.
It was proved by Anari--Oveis Gharan--Vinzant \cite{AOV},
Anari--Liu--Oveis Gharan--Vinzant \cite{ALOVI,ALOV} and Br\"{a}nd\'{e}n--Huh \cite{BH1,BH2}
that $f_M$, $P_M$ and $\overline P_M$ are log-concave on the positive orthant,
equivalently, the Hessian matrix $H_{f_M}|_{\xx=\aaa}$ (resp.\ $H_{P_M}|_{\xx=\aaa}$ and $H_{\overline P_M}|_{\xx=\aaa}$) has exactly one positive eigenvalue
for any $\aaa \in \R^n_{>0}$ (resp.\ $\aaa \in \R^{n+1}_{>0}$).
The log-concavity of these polynomials has important applications to combinatorial properties of matroid.
Let $M$ be a matroid on $[n]$ of rank $r \geq 2$.
We write $\calB_i(M):=\{B\in\calB(M) \mid i\in B\}$ and $\calB_{ij}(M):=\{B\in\calB(M) \mid \{i,j\} \subset B\}$.
Also, we write $I_k(M)$ for the number of independent sets of size $k$ of $M$ and $\widetilde I_k(M)=I_k(M)/{n \choose k}$.
The log-concavity of $f_M$ and $\overline P_M$ is known to imply the following combinatorial inequalities.
\begin{itemize}
\item[($*$)]
$|\calB(M)| \times |\calB_{ij}(M)|
\leq 2 (1-\frac 1 r) |\calB_i(M)| \times | \calB_{j}(M)|$ for all $i,j\in [n]$;
\item[($**$)]
$\widetilde I_{k-1}(M) \times\widetilde I_{k+1}(M) \leq \big(\widetilde I_k(M)\big)^2$ for all $k \geq 1$.
\end{itemize}
(See \cite[Theorem 5]{HSW} and \cite[Remark 15]{HW} for ($*$) and see \cite[Theorem 1.2]{ALOV} and \cite[Corollary 7]{BH1} for ($**$).)
Note that the latter inequality was known as strong Mason's conjecture.

The purpose of this paper is to study when $f_M$ and $\overline P_M$ are strictly log-concave, and determine when equality holds in $(*)$ and $(**)$.
Our main result is the following.

\begin{theorem}
\label{thm:1-1}
Let $M$ be a simple matroid on $[n]$ of rank $r \geq 2$.
\begin{itemize}
\item[(i)]
The Hessian matrix $H_{f_M}|_{\xx=\aaa}$ has signature $(+,-,\dots,-)$ for any $\aaa \in \mathbb R_{>0}^n$,
in particular, $f_M$ is strictly log-concave on $\R_{>0}^n$.
\item[(ii)]
If $M$ is not a uniform matroid, then $H_{\overline P_M}|_{\xx=\aaa}$ has signature $(+,-,\dots,-)$ for any $\aaa=(a_0,a_1,\dots,a_n) \in \mathbb R^{n+1}$ with $a_0 \geq 0$ and $a_1,\dots,a_n>0$.
\end{itemize}
\end{theorem}

Recall that, for a matroid $M$ on $[n]$, the \textit{girth} of $M$ is the minimum cardinality of its circuit, equivalently,
$\mathrm{girth}(M)=\min\{k \mid I_k(M) \ne {n \choose k}\}$.
Theorem \ref{thm:1-1} gives the following combinatorial consequences relating ($*$) and $(**)$.

\begin{corollary}
\label{cor:main}
Let $M$ be a (not necessary simple) matroid on $[n]$ of rank $\geq 2$.
\begin{itemize}
\item[(i)]
Let $i, j \in [n]$ be non-loops.
Then $|\calB(M)| \times |\calB_{ij}(M)|
= 2 (1-\frac 1 r) |\calB_i(M)| \times |\calB_{j}(M)|$ if and only if $i$ and $j$ are not parallel and $M$ has exactly two parallel classes.
\item[(ii)]
$\widetilde I_{k-1}(M) \times \widetilde I_{k+1}(M) = \big(\widetilde I_k(M)\big)^2$
if and only if  $k+1 < \mathrm{girth}(M)$.
\end{itemize}
\end{corollary}

The if part of the above corollary is easy.
Indeed, if $M$ has exactly two parallel classes, then
$r=2$ and $\calB(M)=\{\{x,y\}\mid x\in X,\ y \in Y\}$ for some disjoint sets $X$ and $Y$,
so $|\calB(M)|\times |\calB_{ij}(M)|=|\calB_{i}(M)|\times |\calB_{j}(M)|=|X|\times |Y|$ when $i \in X,$ $j \in Y$.
Also, if $k+1 < \mathrm{girth}(M)$, then
$\widetilde I_{k-1}(M)=\widetilde I_{k}(M)=\widetilde I_{k+1}(M)=1$.
Note also that, if $i$ is a loop of $M$, then $\calB_i(M)=\calB_{ij}(M)=\emptyset.$

The strictness of the log-concavity of $f_M$ was studied by the second and the third author in their previous paper \cite{NY},
where the statement (i) was proved for graphic matroids using the theory of prehomogenous vector spaces.
Our proof in this paper is based on relations between the strong Lefschetz property, the Hodge--Riemann relation, and the Lorentzian property introduced in \cite{BH2}.

This paper is organized as follows: In section 2,
we discuss properties of matroids and their generating polynomials.
In section 3, we discuss relations between the strong Lefschetz property, the Hodge--Riemann relation and the Lorentzian property.
In section 4, we prove our main results. Finally, in section 5, we consider a generalization of Theorem \ref{thm:1-1} to morphism of matroids.

\subsection*{Acknowledgements} The authors wish to express their gratitude to Yasuhide Numata for fruitful discussions. The research of the first author is partially supported by KAKENHI 16K05102, and the research of the second author is partially supported by Grant--in--Aid for JSPS Fellows 19J11207.

\section{Matroids and their generating polynomials}\label{sec:matroids}

We first introduce some notation and terminology on matroids.
We refer the readers to \cite{Ox} for basic properties of matroids.
A {\it matroid} on $[n]$  is an ordered pair $M=([n], \calB(M))$ consisting of finite set $[n]$ and a non-empty collection $\calB(M)$ of subsets of $[n]$ satisfying the following property: 
\begin{itemize}
\item[] If $B_1,B_2 \in \calB(M)$ and $x\in B_1\setminus B_2$, then there is a $y\in B_2\setminus B_1$\\ such that $(B_1\setminus\{x\}) \cup \{y\}\in\calB(M)$.
\end{itemize}
An element of $\calB(M)$ is called a \textit{basis} of $M$
and a subset of a basis of $M$ is called an \textit{independent set} of $M$. We denote by $\calI(M)$ the set of independent sets of $M$. 
It is known that each basis has the same cardinality. The \textit{rank} of a subset $X \subset [n]$ in $M$ is the maximum of the cardinality of independent subsets in $X$ and is denoted by $\rank X$ or $\rank_M X$. For any subset $X\subset [n]$, we define its {\it closure} as $\langle X\rangle:=\Set{i\in[n] \ | \ \rank (X\cup\{i\})=\rank X}$. We call $F\subset [n]$ a {\it flat} of $M$ if $F=\langle F\rangle$.
A subset of $[n]$ which is not an independent set is called a \textit{dependent set} of $M$. 
A minimal dependent set of $M$ is called a \textit{circuit} of $M$.
A circuit having cardinality $k$ is called a $k$-circuit.  
In particular $1$-circuit is called a \textit{loop}. 
We call an element $e$ a {\it coloop} of $M$ if it is contained in each basis of $M$. 
Also, if two elements $e_1$ and $e_2$ form a 2-circuit, then we call $e_1$ and $e_2$ are \textit{parallel}.  
We say that a matroid $M$ is \textit{loopless} (resp.\ \textit{simple}) if it has no loops (resp.\ no loops and no parallel elements). 

\begin{example}
\label{ex:uniform}
Let $\calB$ be the collection of $r$-element subsets of $[n]$, where $r\leq n$. Then $([n], \calB)$ is a matroid of $\rank r$ denoted by $U_{r, n}$. This matroid is called the \textit{uniform matroid} of rank $r$ on an $n$-element set. It is known and easily checked by definition that all rank 2 simple matroids are uniform matroids. 
\end{example}

Let $M=([n], \calB(M))$ be a matroid.
For $e \in [n]$ which is not a loop of $M$, we define the matroid $M/e$ on $[n]\setminus\{e\}$ by $\calB(M/e):=\{B\setminus\{e\} \ | \ e\in B\in\calB(M)\}$, which is called the {\it contraction} of $M$ w.r.t.\ $e$. 
Also, for $X \subset [n]$, we define the matroid $M|_X$ on $[n]\setminus X$ by $\calB(M|_X):=\{B\in\calI(M) \ | \ B\subset X,\ |B|=\rank(X)\}$, which is called 
the \textit{restriction} of $M$ to $X$.
In particular, for $e \in [n]$, we write $M \setminus e=M|_{[n]\setminus \{e\}}$ and call it the \textit{deletion} of $e$ from $M$.


For a matroid $M$ on $[n]$, there is a unique partition $[n]=E_0\sqcup E_1\sqcup\cdots\sqcup E_s$, called the {\em parallel class decomposition}, such that $E_0$ consists of all loops and that $i,j \in [n]$ are parallel if and only if they belongs to the same $E_k$,
where $\sqcup$ denotes a disjoint union.
We call $E_1,\dots,E_s$ {\em parallel classes} of $M$.
Recall that, for a matroid $M=([n],\calB(M))$,
its \textit{simplification} $\overline M$ is the matroid obtained from $M$ by deleting all loops and deleting all but one element in each parallel class in the matroid $M$. We also define the {\it truncation} $TM=([n], \calB(TM))$ by $\calB(TM)=\{ I \in \calI(M) \mid |I|=\rank(M)-1\}$,
and inductively define $T^k M:=T(T^{k-1}M)$ for $k>1$.

Below we write some obvious properties of basis generating polynomials and independent set generating polynomials.
In the rest of this paper, we write $\partial_i=\frac {\partial} {\partial x_i}$.

\begin{lemma}
\label{lem:basicproperty}
Let $M$ be a matroid on $[n]$ of rank $r$. 
\begin{itemize}
\item [(i)] If $i \in [n]$ is a loop, then $\partial_i f_{M}=\partial_i P_{M}=0$.
\item [(ii)] If  $i \in [n]$ is not a loop, then $\partial_i f_{M}=f_{M/i}$ and $\partial_i P_{M}=P_{M/i}$.
\item [(iii)]If $i_1, i_2 \in [n]$ are parallel, then $\partial_{i_1} f_{M} = \partial_{i_2} f_{M}$ and $\partial_{i_1} P_{M} = \partial_{i_2} P_{M}$. Moreover, if $[n]=E_0\sqcup E_1\sqcup\cdots\sqcup E_s$ is the parallel class decomposition,
then 
\[
\textstyle
f_{M}=f_{\overline{M}}\left(\sum_{i\in E_1}{x_i}, \ldots, \sum_{i\in E_s}{x_i}\right)\]
and 
\[
\textstyle P_{M}=x_0^{n-s}{P}_{\overline{M}}\left(x_0, \sum_{i\in E_1}{x_i}, \ldots, \sum_{i\in E_s}{x_i}\right),\]
where $E_0$ is the set of loops and we consider that $\overline M$ is a matroid on $[s]$ such that $i$ corresponds to an element in $E_i$ for $i=1,2,\dots,s$.
\end{itemize}
\end{lemma}

If the Hessian matrix of a polynomial $f \in \mathbb R[x_1,\dots,x_n]$ is non-singular at some $\aaa \in \mathbb R^n$,
then the polynomials $\partial_1 f,\dots,\partial_n f$ must be $\mathbb R$-linearly independent.
In the rest of this section, to prove Theorem \ref{thm:1-1}, we first prove this weaker property.

We need the following combinatorial property of flats of matroids. See \cite[Section 1.4, Exercise 11]{Ox}.

\begin{lemma}
\label{lem:flat}
Let $M$ be a matroid on $[n]$ and $F$ a flat of $M$. If $\{G_1, \ldots, G_\ell\}$ is the set of minimal flats of $M$ that properly contain $F$, then
$[n]\setminus F=\bigsqcup_{i=1}^\ell{(G_i\setminus F)}$.
\end{lemma}

Also, we often use the following elementary fact.

\begin{lemma}
Let $n \geq 2$ and $a_0,a_1,\dots,a_n \in \mathbb R$.
If $\sum_{j \ne k} a_j =a_0$ for all $k=1,2,\dots,n$,
then $a_1=a_2=\cdots=a_n= \frac 1 {n-1}a_0$.
\label{lem:linearalgebra}
\end{lemma}

\begin{proof}
Let $J$ be the all $1$ matrix of size $n$ and $E$ the identify matrix of size $n$.
Then the matrix $J-E$ is non-singular and $(a_1,\dots,a_n)$ must be the unique solution of the system of linear equations $(J-E) \cdot {}^t(x_1,\dots,x_n)= {}^t (a_0,a_0,\dots,a_0)$.
\end{proof}

The following is the main result of this section.

\begin{theorem}
\label{thm:LD}
Let $M$ be a simple matroid on $[n]$ of rank $r \geq 2$.
\begin{itemize}
\item[(i)]
$\partial_1 f_M,\dots,\partial_n f_M$ are $\mathbb R$-linearly independent.
\item[(ii)]
If $M \ne U_{r,n}$ then $\partial_0 \overline P_M,\partial_1 \overline P_M,\dots,\partial_n \overline P_M$ are $\mathbb R$-linearly independent.
\end{itemize}
\end{theorem}

\begin{proof}
(i)
Suppose $(a_1 \partial_1 + \cdots +a_n \partial_n)f_M=0$ for some $a_1,\dots,a_n \in \mathbb R$.
We will prove $a_1= \cdots =a_n=0$.
To do this, we actually prove the following statement using decent induction on the rank of flats.
\begin{align}
\label{2--1}
\sum_{j \in [n] \setminus F} a_j=0 \ \ \ \mbox{ for all flats $F \ne [n]$ of $M$.}
\end{align}
Note that \eqref{2--1} and Lemma \ref{lem:linearalgebra}
imply $a_1=\cdots =a_n=0$ since the equations for rank $1$ flats tell $\sum_{j\ne k} a_j=0$ for all $k=1,2,\dots,n$.

We first prove \eqref{2--1} when $F$ has rank $r-1$.
Let $I \in \calI(M)$ be an independent set such that $\langle I \rangle =F$.
Then $|I|=r-1$ and the coefficient of $\prod_{i \in I} x_i$ in $(a_1\partial_1+ \cdots +a_n \partial_n)f_M$ is
$$\sum_{j \not \in I,\ \!\{j\} \cup I \in \calB(M)} a_j = \sum_{j \in [n] \setminus F}a_j,$$
which must be zero since we assume $(a_1\partial_1+ \cdots +a_n \partial_n)f_M =0
$.

Now suppose $F$ has rank $<r-1$ and assume that \eqref{2--1} holds for all flats  $G$ that properly contain $F$.
Let $G_1,\dots,G_\ell$ be the minimal flats that properly contains $F$.
Note that $\ell \geq 2$ since, by Lemma \ref{lem:flat}, $\ell=1$ implies $G_1=[n]$ and $\rank (F)= \rank([n])-1=r-1$.
Since $[n] \setminus F=\sqcup_{k=1}^\ell (G_k \setminus F)$ by Lemma \ref{lem:flat},
we have
\begin{align*}
\textstyle
\ell \left( \sum_{j \in [n] \setminus F} a_j \right)
& \textstyle
= \sum_{k=1}^\ell \left\{ \left( \sum_{j \in [n] \setminus G_k} a_j \right) + \left( \sum_{j \in G_k \setminus F} a_j \right) \right\}\\
&\textstyle
=\sum_{k=1}^\ell  \left( \sum_{j \in G_k \setminus F} a_j \right)=\sum_{j \in [n] \setminus F} a_j,
\end{align*}
where we use the induction hypothesis to the second equality.
As $\ell \geq 2$, the above equation implies \eqref{2--1} for $F$, as desired.

(ii)
Let $f_k= \sum_{I \in \calI(M),\ |I|=k} (\prod_{i \in I} x_i)$ for $k=0,1,2,\dots,r$, where $f_0=1$.
Then $P_M=x_0^n+x_0^{n-1}f_1+ \cdots +x_0^{n-r}f_r$ and
$$\overline P_M= \frac {n!} {r!} x_0^r + \frac {(n-1)!} {(r-1)!} x_0^{r-1} f_1 + \frac {(n-2)!} {(r-2)!} x_0^{r-2} f_2 + \cdots + (n-r)! f_r.$$
Suppose $(a_0\partial_0+a_1\partial_1+ \cdots +a_n \partial_n)\overline P_M=0$
with $a_0,a_1,\dots,a_n \in \mathbb R$.
We will prove $a_0=a_1=\cdots=a_n=0$ or $M=U_{r,n}$.
Since
\begin{align*}
&(a_0\partial_0+ \cdots +a_n \partial_n)\overline P_M\\
&= \sum_{k=1}^r \frac {(n-k)!} {(r-k)!} \big\{ (n-k+1)a_0 f_{k-1} + (a_1 \partial_1+\cdots+a_n \partial_n) f_k \big\}x_0^{r-k}
\end{align*}
equals to zero,
we have
\begin{align}
\label{2--2}
(n-k+1)a_0 f_{k-1} + (a_1\partial_1+ \cdots+a_n\partial_n) f_k=0
\ \ \mbox{ for }k=1,2,\dots,r.
\end{align}
Since $M$ is simple, $f_1=\sum_{k=1}^n x_k$ and $f_2=\sum_{1 \leq i<j \leq n} x_ix_j$,
so by considering \eqref{2--2} when $k=2$ we have
$$
\sum_{k=1}^n \left\{(n-1)a_0 +{ \sum_{j \ne k} a_j } \right\} x_k=0
\Leftrightarrow
\sum_{j \ne k} a_j = -(n-1)a_0 \ \mbox{ for }k=1,2,\dots,n.$$
This tells $a_1=a_2=\cdots=a_n=-a_0$ by Lemma \ref{lem:linearalgebra}.

If $a_0=0$, then we have $a_0=\cdots=a_n=0$.
Suppose $a_0 \ne 0$.
Then, by substituting $x_1=\cdots=x_n=1$ in \eqref{2--2}, we get
$$a_0 \{ (n-k+1) I_{k-1}(M) -k I_k(M) \}=0$$
for $k=1,2,3,\dots,r$.
This proves
$$I_k(M)= \frac {n-k+1} k I_{k-1}(M) = \cdots = \frac {(n-k+1) (n-k+2)\cdots n} {k!}= {n \choose k}$$
for $k=1,2,\dots,r$,
which tells $M=U_{r,n}$.
\end{proof}

If  $\partial_0 f,\dots,\partial_n f$ are $\mathbb R$-linearly dependent, then so do $\partial_0(\partial_0f),\dots,\partial_n (\partial_0 f)$.
Thus the conclusion of Theorem \ref{thm:LD}(ii) also holds for $P_M$.
Also, for a uniform matroid $U_{r,n}$,
it is easy to see $(-\partial_0+\partial_1+ \cdots+\partial_n)P_{U_{r,n}}=0$,
so the statement (ii) does not hold for uniform matroids.

\section{SLP, HRR and Lorentzian polynomials}\label{sec:HRR}

In this section we discuss relations between the strong Lefschetz property,
the Hodge--Riemann relation, and
Lorentzian polynomials introduced by Br\"{a}nd\'{e}n and Huh \cite{BH2}.
 
\subsection{Lorentzian polynomials}

A polynomial $f \in S$ is said to be \textit{log-concave} (resp.\ \textit{strictly log-concave}) on an open convex set $X \subset \R^n$ if the log of $f$ is a concave (resp.\ strictly concave) function on $X$.
By a well-known criteria for the concavity,
$\log f$ is concave on $X$ if and only if the Hessian matrix of $\log f$ is negative semidefinite at $\xx=\aaa$ for any $\aaa \in X$,
and $\log f$ is strictly concave on $X$ if the Hessian matrix of $\log f$ is negative definite at $\xx=\aaa$ for any $\aaa \in X$.
Note that when $f(\aaa)>0$
the log of $f$ is negative semidefinite (resp.\ negative definite) at $\xx=\aaa$ if and only if $H_f|_{\xx=\aaa}$ has exactly one positive eigenvalue (resp.\ has signature $(+,-,\dots,-)$).
See \cite[Proposition 5.6]{BH2} or \cite[\S 2.3]{NY}.
We simply say that $f$ is log-concave at $\aaa \in \mathbb R^n$ if the Hessian matrix $H_f|_{\xx=\aaa}$ has exactly one positive eigenvalue.

\begin{definition}\label{def:Lorentzian}
Let $f\in\R_{\geq0}[x_1, \ldots, x_n]$ be a homogeneous polynomial of degree $\geq 2$. We call that $f$ is a {\it Lorentzian polynomial} if for any $(k_1, \ldots, k_n)\in\Z_{\geq0}^n$ with $\sum_{i=1}^n k_i\leq \deg f-2$, $\partial_1^{k_1}\cdots\partial_n^{k_n} f$ is identically zero or log-concave at any $\aaa \in \R_{>0}^n$.
\end{definition}

The above property is also known as the strong log-concavity \cite{G}, but we call it Lorentzian since it is equivalent to the Lorentzian property defined in \cite[Definition 2.1]{BH2}.
See \cite[Theorem 5.3]{BH2}.
We note the next observation that follows from the continuity of eigenvalues.

\begin{lemma}
\label{lem:obs}
If $f \in \R_{\geq0}[x_1, \ldots, x_n]$ is Lorentzian,
 then $H_f|_{\xx=\aaa}$ has at most one positive eigenvalue for any $\aaa \in \R_{\geq 0}^n$.
\end{lemma}

An important instance of Lorentzian polynomials are generating polynomials of matroids.
Indeed, the following result is proved in 
\cite{AOV,ALOVI,ALOV,BH1,BH2} (see \cite[Theorem 25]{AOV} and \cite[Theorem 4.1]{ALOV}).

\begin{lemma}
For any matroid $M$ of rank $\geq 2$, the polynomials $f_M$ and $P_M$ are Lorentzian.
\end{lemma}

\subsection{The Strong Lefschetz property and the Hodge--Riemann relation}
Lorentzian polynomials are related to algebraic properties called the strong Lefschetz property and the Hodge--Riemann relation.

Let $S=\R[\partial_1,\dots, \partial_n]$ be the polynomial ring whose variables are $\partial_1,\dots,\partial_n$.
For a homogenous polynomial $f \in \R[x_1, \ldots, x_n]$ of degree $d$, we define the $\R$-algebra
\begin{align*}
R^*_f&:=S/\Ann_S(f),
\end{align*}
where $\Ann_S(f)=\{D \in S \mid Df=0\}$. It is well-known that $R^*_f$ is a Poincar\'{e} duality algebra, that is,
$R_f^d \cong \R$ and the bilinear pairing induced by the multiplication
$R^k_f \times R^{d-k}_f \to R_f^d$ is nondegenerate for all $k$ (see e.g.\ \cite[Theorem 2.1]{MW}).
We say that $R^*_f$ (or $f$) has the \textit{strong Lefschetz property} at degree $k \leq d/2$ (shortly $\SLP_k$) w.r.t.\ a linear form $\ell \in S$ if the multiplication map 
\[\times \ell^{d-2k} : R_f^k \to R_f^{d-k}\]
is an isomorphism. We say that $R^*_f$ (or $f$) satisfies the \textit{Hodge--Riemann relation} at degree $k$ (shortly $\HRR_k$) w.r.t.\ 
a linear form $\ell \in S$
if $R_f^*$ has the $\SLP_k$ w.r.t.\ $\ell$ and the bilinear form 
\[Q_{\ell}^{k} : R_f^{k}\times R_f^{k}\to \R,\ (\xi_1,\xi_2) \mapsto (-1)^{k}[\xi_1\ell^{d-2k}\xi_2]\]
is positive definite on the kernel of $\times\ell^{d+1-2k}: R^k_f \to R^{d-2k+1}_f$,
where $[-] : R_f^d \to \R$ is the isomorphism defined by $D \mapsto D(\partial_1,\dots,\partial_n)f$.

We are actually only interested in $\SLP_1$ and $\HRR_1$ in this paper.
For $\aaa=(a_1,\dots,a_n) \in \R^n$,
we write $\ell_\aaa=a_1\partial_1+ \cdots + a_n \partial_n$.

\begin{lemma}
\label{The Hessian criterion}
Let $f \in S$ be a homogeneous polynomial of degree $\geq 2$ and $\aaa \in \R^n$. Assume that $f(\aaa) > 0$. Then,
\begin{enumerate}
\item[(i)] $R_f$ has the $\SLP_1$ w.r.t.\ $\ell_{\aaa}$ $\Leftrightarrow$ $Q_{\ell_\aaa}^1$ is non-singular.
\item[(ii)] $R_f$ has the $\HRR_1$ w.r.t.\ $\ell_{\aaa}$ $\Leftrightarrow$ $-Q_{\ell_\aaa}^1$ has signature $(+,-,\dots,-)$.
\end{enumerate}  
\end{lemma}

\begin{proof}
The statement (i) is obvious. We show (ii).
Define the map $\psi_\aaa: R^1_f \to R^d_f$ by $\psi_\aaa(h)=\ell_\aaa^{d-1} h$.
Since the map 
\[
\times \ell_{\bm{a}}^{d} : R_f^0\xrightarrow{\times \ell_{\bm{a}}} R_f^1\xrightarrow{\psi_\aaa} R_f^d\] is an isomorphism,
the decomposition $R_f^1=\R \ell_{\bm{a}}\oplus\Ker \psi_\aaa$ is orthogonal with respect to $Q_{\ell_{\bm{a}}}^1$. Since $-Q_{\ell_{\bm{a}}}^1(\ell_{\bm{a}}, \ell_{\bm{a}})=[\ell_{\bm{a}}^d]=d!f(\bm{a})>0$, it follows that $R_f$ satisfies the $\HRR_1$ w.r.t.\ $\ell_{\bm{a}}$ if and only if $-Q_{\ell_{\bm{a}}}^1$ is nondegenerate and has only one positive eigenvalue.
\end{proof}

The previous lemma implies the following fact.

\begin{lemma}
\label{prop:SLPHRR}
If $f \in \R_{\geq 0}[x_1,\dots,x_n]$ is Lorentzian,
then for any $\aaa \in R_{\geq 0}^n$ with $f(\aaa) > 0$,
$f$ has the $\SLP_1$ w.r.t.\ $\ell_\aaa$ if and only if $f$ has the $\HRR_1$ w.r.t.\ $\ell_\aaa$.
\end{lemma}

\begin{proof}
The Hessian matrix $H_{f}|_{\bm{x}=\bm{a}}$ is (a positive scalar multiple of) the representation of the
symmetric bilinear form $-{{Q_{\ell_{\bm{a}}}^1}}:R_f^1 \times R_f^1 \to \R$ w.r.t.\ the generating set $\{\partial_1,\dots,\partial_n\}$ of $R_f^1$.
Indeed, by definition, we have 
\begin{align*}
{-Q_{\ell_{\bm{a}}}^1}(\partial_i, \partial_j) ={[\partial_i \ell_{\bm{a}}^{d-2}\partial_j]}=\left(a_1 \partial_1+\cdots+a_n\partial_n \right)^{d-2}\left(\partial_i \partial_jf\right)={(d-2)!}(\partial_i \partial_j f) | _{\bm{x}=\bm{a}},
\end{align*}
where $d=\deg f$.
Since Sylvester's law tells that the number of positive eigenvalues of the symmetric matrix representing a fixed symmetric bilinear form does not depend on the choice of a generating set,
the number of positive eigenvalues of $-Q_{\ell_{\bm{a}}}$ equals to that of $H_f|_{\xx=\aaa}$.
Then the assertion follows from Lemmas \ref{lem:obs} and \ref{The Hessian criterion}.
\end{proof}

We also note the next fact,
which immediately follows from the fact that $\partial_1,\dots,\partial_n$ is an $\R$-basis of $R_f^1$ if and only if $\partial_1 f ,\dots, \partial_n f$ are $\R$-linearly independent.

\begin{lemma}
\label{lem:HessianSLP}
Let $f \in \R[x_1,\dots,x_n]$ be a homogeneous polynomial of degree $\geq 2$ and $\aaa \in \R^n$.
If $\partial_1 f ,\dots,\partial_n f$ are $\R$-linearly independent,
then $R_f$ has the $\SLP_1$ $($resp.\ $\HRR_1)$ w.r.t.\ $\ell_\aaa \in S$ if and only if $H_f|_{\xx=\aaa}$ is non-singular $($resp.\ has signature $(+,-,\dots,-))$.
\end{lemma}

\subsection{The local HRR and the SLP}
We say that a homogeneous polynomial $f\in\R[x_1, \ldots, x_n]$ of degree $d\geq 2k+1$ has the local $\HRR_k$ w.r.t.\ a linear form $\ell \in S$ if, for any $i=1,2,\dots,n$,
$\partial_i f$ is either zero or has the $\HRR_k$ w.r.t.\ $\ell$.
The next proposition would be known for experts ,
but we include its proof since we cannot find a version which covers the case we need (see e.g., \cite[Proposition 7.15]{AHK} for a similar statement).

\begin{lemma}
\label{prop:local}
Let $f \in \R_{\geq 0}[x_1,\dots,x_n]$ be a homogeneous polynomial of degree $d$ and $k$ a positive integer with $d \geq 2k+1$, and $\aaa=(a_1,\dots,a_n) \in \R^n$.
Suppose that $f$ has the local $\HRR_k$ w.r.t.\ $\ell_\aaa$.
\begin{itemize}
\item[(i)]
If $\aaa \in \R^n_{>0}$, then
$R_f$ has the $\SLP_k$ w.r.t.\ $\ell_\aaa$.
\item[(ii)]
If $a_1=0$, $a_2,\dots,a_n>0$ and $\{\xi \in R_f^k\mid \partial_i \xi=0 \mbox{ for }i=2,\dots,n\}=\{0\}$,
then $R_f$ has the $\SLP_k$ w.r.t.\ $\ell_\aaa$.
\end{itemize}
\end{lemma}

\begin{proof}
We prove (i) and (ii) simultaneously.
Without loss of generality, we may assume that $f\notin\R[x_1, \ldots, \hat{x}_i, \ldots, x_n]$ for any $i$. 
Consider the following two maps:  
\[[-] : R_f^d\xrightarrow{\sim}\R,\ h \mapsto  [h]=h(\partial_1, \ldots, \partial_n)f,\]
\[[-]_i : R_{\partial_if}^{d-1}\xrightarrow{\sim}\R ,\ h' \mapsto [h']_i=h'(\partial_1, \ldots, \partial_n)\partial_if.\]
Also, let $Q_i$ be the Hodge--Riemann bilinear form  for $R_{\partial_if}=S/\Ann_S(\partial_if)$ with respect to $\ell_{\bm{a}}$:  
\[Q_i : R^k_{\partial_if}\times R^k_{\partial_if} \to \R,\ (v,w) \mapsto Q_i(v, w)=(-1)^k [v\ell_{\bm{a}}^{d-2k-1}w]_i. \]

Suppose that $L\in R_f^k$ satisfies $L\ell_{\bm{a}}^{d-2k}=0$ in $R_f^{d-k}$. To prove the desired statement,
what we must prove is $L=0$ under the assumption of (i) or (ii).
Since $L\ell_{\bm{a}}^{d-2k}=0$ in $R_{\partial_i f}^{d-k}$ as well, $L \in R^k_{\partial_i f}$ is contained in the kernel of
\[ \times \ell_\aaa^{d-2k}: R^k_{\partial_i f} \to R^{d-k}_{\partial_i f}.\]
Since $Q_i$ is positive definite on the kernel of the above map, we have 
\begin{equation}\label{2-3}
Q_i(L, L)\geq 0,
\end{equation}
and $Q_i(L, L)=0$ if and only if $L=0$ in $R_{\partial_if}=S/\Ann_S(\partial_if)$.
On the other hand, since $L\ell_{\bm{a}}^{d-2k}=0$ in $R_f^{d-k}$, we have
\[0=[L^2\ell_{\bm{a}}^{d-2k}]=\left[\sum_{i=1}^n{a_i\partial_iL^2\ell_{\bm{a}}^{d-2k-1}}\right]=
\sum_{i=1}^n{a_i[L^2\ell_{\bm{a}}^{d-2k-1}]_i}=(-1)^k\sum_{i=1}^n{a_iQ_i(L, L)}.\]

Now assume $a_i>0$ for all $i$.
We note that 
$\{\xi \in R_f \mid \partial_i \xi=0 \mbox{ for all }i\}=R_f^d$ since, for any $D \in S$ of degree $<d$, if $Df \ne 0$ then $\partial_i (Df) \ne 0$ for some $i$.
The above equation and \eqref{2-3} tell that $Q_i(L, L)=0$ for all $i$, and therefore $L=0$ in $R_{\partial_i f}$ for all $i$.
But, since $\{\xi \in R_f \mid \partial_i \xi=0 \mbox{ for all }i\}=R_f^d$, this implies $L=0$ in $R_{f}$, proving (i).

The proof of (ii) is similar.
Indeed, if $a_1=0$ and $a_2,\dots,a_n>0$, then the same argument tells $Q_i(L, L)=0$ and $L=0$ in $R_{\partial_i f}$ for all $i=2,\dots,n$.
Then $\partial_i L=0$ in $R_f$ for all $i=2,\dots,n$, and the assumption of (ii) tells $L=0$ in $R_f$.
\end{proof}

The following statement immediately follows from Lemmas \ref{prop:SLPHRR} and \ref{prop:local}, both of which are basic,
but is crucial to prove Theorem \ref{thm:1-1}(i).

\begin{theorem}
\label{thm:LorentzianHRR}
If $f \in \R[x_1,\dots,x_n]$ is Lorentzian, then $f$ has the $\HRR_1$ w.r.t.\ $\ell_\aaa$ for any $\aaa \in \R^n_{>0}$.
\end{theorem}

\begin{proof}
By Lemma \ref{prop:SLPHRR}, we only have to show that $R_f$ has the $\SLP_1$ w.r.t.\ $\ell_\aaa$. We prove by induction on $d=\deg f$. When $d=2$, this is trivial since any degree $2$ homogeneous polynomial has the $\SLP_1$ w.r.t.\ any linear form by definition.
When $d\geq3$, by Lemma \ref{prop:local}, it suffices to show that for each $i$  with $\partial_i f \ne 0$, $\partial_i f$ satisfies the $\HRR_1$ w.r.t.\ $\ell_\aaa$. Since $\partial_i f$ is also a Lorentzian polynomial if it is non-zero  by the definition of the Lorentzian property, the claim is trivial by the induction hypothesis. 
\end{proof}

\begin{remark}
Maeno--Numata \cite{MN} conjectured that, for any matroid $M$, $R_{f_M}$ has the $\SLP_k$ for all $k$ w.r.t.\ some linear form $\ell$.
Since $f_M$ is Lorentzian, the above statement verifies this conjecture when $k=1$.
\end{remark}

\section{Proof of main results}

In this section, we prove Theorem \ref{thm:1-1} and Corollary \ref{cor:main} in the introduction.
We first prove Theorem \ref{thm:1-1}.
Since $f_M$ and $\overline P_M$ are Lorentzian,
the statement (i) and the statement (ii) when $a_0 \ne 0$ immediately follow from
Theorems \ref{thm:LD} and \ref{thm:LorentzianHRR} together with Lemma \ref{lem:HessianSLP}.
Then the next statement completes the proof of Theorem \ref{thm:1-1}.

\begin{theorem}
\label{thm:main}
Let $M$ be a matroid on $[n]$ of rank $r \geq 2$
and $\aaa=(a_1,\dots,a_n) \in \R^n_{>0}$.
Then $\overline P_M$ 
has the $\HRR_1$ w.r.t.\ $\ell_\aaa=a_1\partial_1+\cdots+a_n \partial_n$.
\end{theorem}

\begin{proof}
To prove this, we may assume that $M$ is loopless.
Also, it suffices to prove that $\overline P_M$ has the $\SLP_1$ with respect to $\ell_\aaa$ since $\SLP_1$ and $\HRR_1$ are equivalent in this case by Lemma \ref{prop:SLPHRR}.
We prove that $\overline P_M$ has the $\SLP_1$ w.r.t.\ $\ell_\aaa$ by using induction on $r$.

If $M$ has rank $2$, then the assertion is obvious because any degree $2$ homogeneous polynomial has the $\SLP_1$.

Suppose that $M$ has rank $r \geq 3$.
Since $\partial_0 \overline P_M=\overline P_{TM}$
and $\partial_i \overline P_M=\overline P_{M / i}$, the induction hypothesis and Lemma \ref{prop:local}(ii) tell that
it suffices to prove
\begin{align}
\label{socle}
\{ L \in R_{\overline P_M}^1 \mid \partial_i L=0
\mbox{ for }i=1,2,\dots,n\} =\{0\}.
\end{align}
Note that $R_{\overline P_M}$ is a quotient ring of $\R[\partial_0,\dots,\partial_n]$. Let $L=b_0\partial_0+b_1\partial_1+ \cdots +b_n \partial_n$ and assume $\partial_i L \overline P_M=0$ for all $i=1,2,\dots,n$.
To prove \eqref{socle}, what we must prove is $L \overline P_M=0$.

Since  $\partial_i L \overline P_M=0$ for all $i=1,2,\dots,n$,
we have
\begin{align*}
L \overline P_M= c x_0^{r-1}
\end{align*}
for some $c \in \R$.
Let $f_k=\sum_{I \in \calI(M),\ |I|=k} (\prod_{i \in I} x_i)$ for $k=1,\dots,r$.
Then $\overline P_M=\frac {n!} {r!}x_0^r+\frac {(n-1)!} {(r-1)!}x_0^{r-1}f_1+\frac {(n-2)!} {(r-2)!}x_0^{r-2}f_2 + \cdots$, so $L\overline P_M$ is the polynomial of the form
{\small
$$
\frac {(n-1)!} {(r-1)!} \left(nb_0+\sum_{k=1}^n b_k \right)x_0^{r-1} +\frac {(n-2)!} {(r-2)!}
\left( (n-1) b_0 f_1+ \left( \sum_{k=1}^n b_k \partial_k \right)f_2 \right) x_0^{r-2}+ \cdots.$$
}

\noindent
Since $f_1= \sum_{k=1}^n x_k$ and $f_2=\sum_{1 \leq i<j \leq n} x_ix_j$,
comparing coefficients of $x_0^{r-1}$ and $x_0^{r-2}$ in $L\overline P_M=c x_0^{r-1}$, we have
$$
c= nb_0+b_1+\cdots +b_n
$$
and
$$
\textstyle\sum_{k=1}^n \left( (n-1)b_0+ \sum_{ j \ne k} b_j \right) x_k=0.$$
Then we have $\sum_{j \ne k} b_j=-(n-1)b_0$ for all $k=1,2,\dots,n$,
and therefore $b_1=\cdots=b_n=-b_0$ by Lemma \ref{lem:linearalgebra}.
This implies 
$c=nb_0+b_1+ \cdots +b_n=0,$
and $L \overline P_M=cx_0^{r-1}=0$ as desired.
\end{proof}

In the rest of this section,
we prove Corollary \ref{cor:main}.
We need the following two technical lemmas.

\begin{lemma}
\label{lem:technical1}
Let  $f \in \R[x_1,\dots,x_n]$ be a homogeneous polynomial having the $\HRR_1$ w.r.t.\ $\ell_\aaa$ with $\aaa \in \R^n$ and let $\ell_1,\ell_2 \in S$ be linear forms.
If $\ell_1$ and $\ell_2$ are $\R$-linearly independent in $\R_f^1$ and $(\ell_1\ell_1 f)(\aaa) >0$, then
$$\det \begin{pmatrix}
(\ell_1\ell_1 f)(\aaa) & (\ell_1  \ell_2 f) (\aaa)\\
(\ell_1\ell_2 f)(\aaa) & (\ell_2  \ell_2 f) (\aaa)
\end{pmatrix}<0.$$
\end{lemma}

\begin{proof}
Consider the bilinear form
$$Q:R_f^1 \times R_f^1 \to \R, \ \ \ \ (\xi_1,\xi_2) \mapsto \xi_1\xi_2 \ell_{\aaa}^{d-2}  \cdot f,$$
where $d=\deg f$.
Observe that the subspace $W=\mathrm{span}_\R \{\ell_1,\ell_2\} \subset R_f^1$ has $\R$-dimension $2$ by the assumption.
Since $f$ has the $\HRR_1$ w.r.t.\ $\ell_\aaa$,
the bilinear form $Q$ has signature $(+,-,\dots,-)$,
so the restriction of $Q$ to $W$ has signature $(+,-),(0,-)$, or $(-,-)$ by Cauchy's interlacing theorem (see \cite[Lemma 2.4]{AOV}).
Since $Q(\ell_1,\ell_1)=(d-2)! (\ell_1 \ell_1 f)(\aaa)>0$, the latter two cases cannot occur.
Then, since the determinant in the statement is a representation matrix of the bilinear form $Q|_W:W \times W \to \R$ (up to a positive scalar multiplication), it must be negative as $Q|_W$ has signature $(+,-)$.
\end{proof}

\begin{lemma}
\label{technical2}
Let $M$ be a matroid on $[n]$ of rank $r \geq 2$
and $\aaa=(a_1,\dots,a_n) \in R_{>0}^n$.
\begin{itemize}
\item[(i)] Let $i,j \in [n]$ be non-loops and assume $\dim_\R R_{f_M}^1 \geq 3$.
If $i$ and $j$ are not parallel, then
$\ell_\aaa,\partial_i,\partial_j$ are $\R$-linearly independent in $R_{f_M}$.
\item[(ii)] If $M \ne U_{r,n}$ then $\partial_0,\ell_\aaa$ are $\R$-linearly independent in $R_{\overline P_M}$.
\end{itemize}
\end{lemma}

\begin{proof}
By Theorem \ref{thm:LD} and Lemma \ref{lem:basicproperty}(iii),
the $\R$-vector space $\{\ell_{\aaa}\in S \mid \ell_a f_M=0\}$ is generated by
$$\{ \partial_k -\partial_{k'} \mid \mbox{$k$ and $k'$ are parallel in $M$} \} \cup \{ \partial_k\mid \mbox{$k$ is a loop of $M$}\}.$$
This vector space has the trivial intersection with the subspace $\mathrm{span}_\R\{\ell_\aaa,\partial_i,\partial_j\} \subset S$,
which guarantees (i).
The proof for (ii) is similar.
\end{proof}

Note that $\dim_\R R_{f_M}^1$ equals to the number of parallel classes of $M$,

We now prove Corollary \ref{cor:main}.
It is the special case of the following statement when $\aaa=(1,1,\dots,1)$.

\begin{theorem}
Let $M$ be a matroid on $[n]$ of rank $r \geq 2$, $i,j \in [n]$, and let $f_k=\sum_{I \in \calI(M),\ |I|=k} (\prod_{i \in I} x_i)$ for $k=0,1,2,\dots,r$.
\begin{itemize}
\item[(i)] If $i$ and $j$ are non-loops and $M$ has at least three parallel classes,
then for any $\aaa \in \R^n_{>0}$ one has
$$\big(f_M(\aaa)\big) \times \big(\partial_i\partial_j f_M(\aaa)\big)
< 2\left(1-\frac 1 r\right) \big(\partial_i f_M(\aaa)\big) \times \big(\partial_j f_M(\aaa)\big).$$
\item[(ii)]
For any $\aaa \in \R^n_{>0}$ and $k+1 \geq \mathrm{girth}(M)$, one has
$$ \frac {f_{k-1}(\aaa)} { {n \choose k-1}}
\frac {f_{k+1}(\aaa)} { {n \choose k+1}}
<
\left(\frac {f_{k}(\aaa)} { {n \choose k}}\right)^2.$$
\end{itemize}
\end{theorem}

We note that the non-strict inequalities are known and proved in
\cite{ALOV,BH1,BH2,HSW}. In particular,
our proof of the above theorem is based on the proofs of \cite[Corollary 6]{BH1} and \cite[Lemma 4.4]{BH2}.

\begin{proof}
We first prove (i).
If $i$ and $j$ are parallel in $M$, then $\partial_i\partial_j f_M=0$ so the assertion is obvious. We assume that $i$ and $j$ are not parallel.

To simplify the notation, we write
$$f=f_M(\aaa),\ f_i=(\partial_i \cdot f_M)(\aaa),\ f_j=(\partial_j \cdot f_M)(\aaa)\mbox{ and } f_{ij}=(\partial_i \partial_j \cdot f_M)(\aaa).$$
By Euler's identity, we have
$$(\ell_\aaa^2 \cdot f_M)(\aaa)=r(r-1)f,\ (\ell_a \partial_i \cdot f_M)(\aaa)=(r-1)f_i \mbox{ and }(\ell_a \partial_j \cdot f_M)(\aaa)=(r-1)f_j.$$
Then, for any $t \in \R$, we have
\begin{align}
\nonumber
& \frac 1 {r-1} \det \begin{pmatrix}
\big(\ell_\aaa^2 \cdot f_M\big)(\aaa) & \big(\ell_\aaa (\partial_i+t\partial_j) \cdot f_M\big)(\aaa)\\
 \big(\ell_\aaa (\partial_i+t\partial_j) \cdot f_M \big)(\aaa) & \big((\partial_i+t\partial_j)^2\cdot f_M \big) (\aaa)
\end{pmatrix}\\
\nonumber
&= \frac 1 {r-1} \det \begin{pmatrix}
r(r-1)f & (r-1) (f_i+tf_j)\\
  (r-1) (f_i+tf_j) & 2t f_{ij}
\end{pmatrix}\\
&=-(r-1)f_j^2 t^2 +2(rff_{ij}-(r-1)f_if_j)t-(r-1)f_i^2. \label{det1}
\end{align} 
The discriminant of the quadratic polynomial \eqref{det1} in $t$ is
\begin{align}
\label{disc1}
(rff_{ij}-(r-1)f_if_j)^2-(r-1)^2f_i^2f_j^2
= rf_if_j(rff_{ij}-2(r-1)f_if_j).
\end{align}
Since $\ell_\aaa$ and $\partial_i+t\partial_j$ are $\R$-linearly independent in $R_{f_M}$ by Lemma \ref{technical2}, the determinant in \eqref{det1} is negative for all $t \in \R$ by Lemma \ref{lem:technical1}. This tells that the discriminant \eqref{disc1} must be negative. Hence we have $rff_{ij}<2(r-1)f_if_j$, as desired.

(ii) 
Let $M'=T^{r-k-1}M$.
Since $k+1 \geq \mathrm{girth}(M)$,
$M' \ne U_{r,n}$.
Then, since
$$
\textstyle \overline P_{M'}= (n-k-1)! f_{k+1} + (n-k)! f_{k} \cdot x_0 + \frac {(n-k+1)!} 2 f_{k-1}\cdot x_0^2 + \cdots,$$
we have
\begin{align*}
&\det \begin{pmatrix}
\big ( \ell_\aaa^2 \overline P_{M'} \big) (0,a_1,\dots,a_n) & \big( \ell_\aaa \partial_0 \cdot \overline P_{M'}\big) (0,a_1,\dots,a_n)\\
\big( \ell_\aaa \partial_0 \cdot \overline P_{M'}\big) (0,a_1,\dots,a_n) & \big(\partial_0^2 \cdot \overline P_{M'} \big) (0,a_1,\dots,a_n)
\end{pmatrix}\\
&=
\det \begin{pmatrix}
(n-k-1)! \big(\ell_\aaa^2 f_{k+1} \big)(\aaa) & (n-k)! \big(\ell_\aaa\cdot f_k\big)(\aaa)\\
(n-k)! \big(\ell_\aaa f_k\big)(\aaa) & (n-k+1)! f_{k-1}(\aaa)
\end{pmatrix}\\
&
=(n-k)!(n-k-1)! \det \begin{pmatrix}
(k+1)k f_{k+1}(\aaa) & (n-k) k f_k(\aaa)\\
(n-k) k f_k(\aaa) (\aaa) & (n-k+1) f_{k-1}(\aaa) \end{pmatrix}\\
&=k(n-k)! (n-k-1)! \big\{ (n-k+1)(k+1) f_{k+1}(\aaa)f_{k-1}(\aaa)-(n-k)k (f_k(\aaa))^2\big\}.
\end{align*}
Recall that $\overline P_{M'}$ has the $\HRR_1$ w.r.t.\ $\ell_\aaa=a_1 \partial_1+ \cdots + a_n \partial_n$ by Theorem \ref{thm:main}.
Then the above determinant must be negative by Lemmas \ref{lem:technical1} and \ref{technical2}. Hence we have
$$(n-k+1)(k+1) f_{k+1}(\aaa)f_{k-1}(\aaa)-(n-k)k (f_k(\aaa))^2<0.$$
It is easy to see that this inequality is the same as the desired inequality.
\end{proof}

\section{Morphism of matroids}\label{sec:morphism}

Recently, Eur--Huh \cite{EH} extend the Lorentzian property of $f_M$ and $P_M$ to basis generating polynomials of morphisms of matroids.
In this section, we generalize Theorem \ref{thm:LD} to morphisms of matroids.
Note that, by Theorem \ref{thm:LorentzianHRR}, this partially generalize Theorem \ref{thm:1-1}.

\begin{definition}
\label{def:morphism}
Let $M$ be a matroid on $[n]$ of rank $r$ and $N$ a matroid of rank $r'$.
A {\it morphism} $\vphi : M \to N$ is a map between the underlying space  satisfying the following equivalent conditions:  
\begin{itemize}
\item[(i)] For any $S_1\subset S_2\subset [n]$, we have
\[\rank_N(\vphi(S_2))-\rank_N(\vphi(S_1))\leq \rank_M S_2-\rank_M S_1.\] 
\item[(ii)] For any flat $F$ of $N$, $\vphi^{-1}(F)$ is a flat of $M$.  
\end{itemize}
\end{definition}

We refer the readers to \cite{EH} for basic properties and typical instances of morphisms of matroids.

Let $\vphi : M \to N$ be as in Definition \ref{def:morphism}.
A subset $I\subset [n]$ is a {\it basis} of $\vphi$ if $I$ is an independent set of $M$ and $\langle \vphi(I) \rangle $ equals to the ground set of $N$, equivalently, $\rank(\varphi(I)) = \rank N$. We write $\calB(\vphi)$ for the set of bases of $\vphi$. Also, for $k \geq 0$, we write $\calB(\vphi)_k=\{I \in \calB(\varphi) \mid |I|=k\}$.
We define the {\it basis generating polynomial} $P_{\vphi}$  of $\vphi$ as 
\[P_{\vphi}:=\sum_{I \in \calB(\vphi)} x_0^{n-|I|} \left(\prod_{i \in I} x_i \right).\]
Also, we call
$$\overline P_{\vphi}= \partial_0^{n-r} P_\vphi$$
the {\it reduced basis generating polynomial} of $\vphi$, where $r=\mathrm{rank}(M)$.
Below we give a few remarks on $\calB(\varphi)$ and $P_\varphi$.

\begin{remark}\label{rem:quotient}
Let $\varphi$ be as above. 
\begin{itemize}
\item $P_\varphi$ is non-trivial only when $\varphi([n])$ has rank $r'$ in $N$.
We assume this throughout the paper.
\item
$\calB(\varphi)=\bigsqcup_{k=r'}^r \calB(\varphi)_k$
and $\calB(\varphi)_r=\calB(M)$.
Also, $([n],\calB(\varphi)_k)$ is a matroid for any $r'\leq k \leq r $
(see the remark at the end of \cite[section 2]{EH}).
\item When $r=r'$, then $P_{\vphi}=x_0^{n-r}f_{M}$ and $\overline{P}_{\vphi}=(n-r)! f_{M}$. 
Also, if $N=U_{0,1}$, then we have $P_\vphi=P_{M}$ and $\overline{P}_{\vphi}=\overline{P}_{M}$. From this viewpoint, basis generating polynomials of morphisms can be seen as a generalization of basis generating polynomials and independent set generating polynomials. 
\end{itemize}
\end{remark}

Let $M$ be a matroid on $[n]$.
For any morphism $\vphi : M \to N$, we say that two elements $i$ and $j$ in $[n]$ are {\it $\vphi$-parallel} if $\vphi(i)$ and $\vphi(j)$ are parallel in $N$.
We define $\varphi$-parallel classes in the same way as usual parallel classes.
Also $i \in[n]$ is said to be a {\it $\vphi$-loop} if $\vphi(i)$ is a loop of $N$. We set $L_\vphi:=\{ i \in [n] \mid \mbox{ $i$ is a $\varphi$-loop}\}.$

By \cite[Corollary 22]{EH},
$P_\vphi$ is a Lorentzian polynomial.
Thus its Hessian matrix has signature $(+,-,\dots-)$ when $\partial_0\overline P_\varphi,\partial_1 \overline P_\varphi,\dots,\partial_n \overline P_\varphi$ are 
$\R$-linearly independent.
As the next example shows, this linear independency does not hold for all morphisms.

\begin{example}
\label{exam:linearlydependent}
Let $\varphi:M \to N$ be as in Definition \ref{def:morphism}.
\begin{itemize}
\item[(1)]
If $i$ is a loop of $M$, then $\partial_i P_\varphi=0$.
Similarly, 
if $i$ and $j$ are parallel in $M$,  then $(\partial_i-\partial_j)P_\varphi=0$.
\item[(2)]
If $r= r'$, then $P_\varphi=f_M$.
In this case, $\partial_0 \overline P_\varphi=0$ since $\overline P_\varphi =(n-r)! f_M$ does not contain $x_0$.
\item[(3)]
Suppose that $r-r'=1$ and $L_\varphi=\{1\}$.
Then it is not hard to see
$$P_\varphi=x_0^{n-r}(x_0+x_1) \sum_{I \in \calB(\varphi)_{r'}} \left(\prod_{i\in I} x_i \right ) + x_0^{n-r} \sum_{ 1 \not \in I \in \calB(\varphi)_{r}}  \left(\prod _{i\in I} x_i\right).$$
In this case, $(\partial_0-(n-r+1)\partial_1)\overline P_\vphi=0$.
\item[(4)]
Suppose that $M|_{L_\varphi}$ is a uniform matroid on $ L_\varphi$ and $|[n] \setminus L_\varphi|=r'$.
Then it is not difficult to see
$$\textstyle P_\varphi = P_{M |_{L_\varphi}} \times \left (\prod_{i \in [n] \setminus L_\varphi} x_i \right).$$
(See also Lemma \ref{lem:4.1} below). In this case, $(-\partial_0+ \sum_{i \in L_\varphi} \partial_i) P_{\varphi}=0$.

Here is an instance of such a morphism.
Consider the morphism $\varphi : U_{r-r',n-r'}\bigoplus U_{r',r'} \to N=U_{0,1} \bigoplus U_{r',r'}$ which send elements in $U_{r,n}$ to the loop of $N$ (i.e.\ the element of $U_{0,1}$) and whose restriction to $U_{r',r'}$ is an isomorphism.
This map is indeed a morphism of matroids and satisfies the above condition.
\end{itemize}
\end{example}

We will prove that these are the only cases that the linear dependency of the polynomial $\partial_0 \overline P_\varphi,\dots,\partial_n \overline P_\varphi$ occurs.
For the proof, we need the following lemmas.

\begin{lemma}
\label{lem:4.1}
Let $\varphi: M \to N$ be a morphism of matroids
and $I \in \calB(\varphi)_{\rank(N)}$.
Then $I \cap L_\varphi= \emptyset$ and, for any $J \subset L_\varphi$, one has $I \cup J \in \calB(\varphi)$ if and only if $J \in \calI(M|_{L_\varphi})$.
\end{lemma}

\begin{proof}
Let $I \in \calB(\varphi)_{\rank(N)}$. If $I$ contains a $\varphi$-loop $j$, then $\mathrm{rank}(N)=\mathrm{rank}(\varphi(I))=\mathrm{rank}(\varphi(I \setminus \{j\}))$, so 
$\rank(N) \leq |I \setminus \{j\}| <|I|$, contradicting $|I|=\mathrm{rank}(N)$.
Also, for any $J \subset L_\varphi$, since 
$$
\mathrm{rank}(I \cup J)-\rank(J) \geq \mathrm{rank}(\varphi(I \cup J))-\mathrm{rank} (\varphi(J))=|I|-0,$$
one has $I \cup J \in \calI(M)$ if and only if $\mathrm{rank}(J)=|J|$.
The first condition is equivalent to $I \cup J \in \calB(\varphi)$
since $\rank(\varphi(I \cup J))=\rank(\varphi(I))=\rank(N)$,
and the latter condition is equivalent to $J \in \calI(M|_{L_\varphi})$.
\end{proof}

\begin{lemma}
\label{lem:QQQ}
Let $m$ be a positive integer,
$M$ a simple matroid on $[n]$ of rank $r \geq 2$,
and $f=\partial^{n-r+m} (x_0^m P_M)$.
Then $\partial_0 f,\partial_1 f,\dots,\partial_n f$ are $\R$-linearly independent.
\end{lemma}

\begin{proof}
Let $\ell= \sum_{k=0}^n a_k \partial_k$ with $a_k \in \R$.
Then $\ell f$ is a polynomial of the form
{\small
$$c_0 \left\{ (n+m) a_0+ \sum_{k=1}^n a_k \right\}x_0^{r-1}
+  c_1\left\{ \sum_{k=1}^n \left( (n+m-1) a_0+ \sum_{j \ne k} a_j \right)x_k  \right\}x_0^{r-2}+ \cdots,
$$
}

\noindent
where $c_0=\frac {(n+m-1)!} {(r-1)!}$ and $c_1=\frac {(n+m-2)!} {(r-2)!}$.
Suppose $\ell f=0$.
Then we have 
(i) $(n+m)a_0 + \sum_{k=1}^n a_k=0$
and
(ii) $(n+m-1)a_0+ \sum_{j \ne k} a_j =0$ for $k=1,2,\dots,n$.
The condition (ii) tells $a_j= -\frac {n+m-1} {n-1} a_0$ for all $j$ by Lemma \ref{lem:linearalgebra},
but then condition (i) says $0=(n+m)a_0 -\frac {n(n+m-1)}{ n-1}a_0= -\frac {m}{n-1} a_0$.
Then we have $a_0= \cdots=a_n=0$, so $\partial_0f,\partial_1 f,\dots,\partial_n f$ are linearly independent. 
\end{proof}

Now we prove the main  result of this section.

\begin{theorem}
\label{thm:morphism}
Let $M$ be a simple matroid on $[n]$ of rank $r$, $N$ a matroid of rank $r'$,
and $\varphi:M \to N$ a morphism of matroids such that $\rank_N(\varphi([n]))=r'$.
Then $\partial_0 \overline P_\varphi,\partial_1 \overline P_\varphi,\dots,\partial_n \overline P_\varphi$ are $\R$-linearly dependent if and only if one of the following holds:
\begin{itemize}
\item[(A)] $r=r'$;
\item[(B)] $r-r'=1$ and $|L_\varphi|=1$;
\item[(C)] $M|_{L_\varphi}$ is a uniform matroid and $|[n]\setminus L_\varphi|=r'$.
\end{itemize}
\end{theorem}

\begin{proof}
Let $\ell=\sum_{k=0}^n a_k \partial_k$ be non-zero, where $a_0,\dots,a_n \in \R$,
and assume $\ell \overline P_\varphi=0$.
We prove that $\varphi$ satisfies one of (A), (B) and (C).
To prove this, we may assume $r>r'$.
Recall that
\begin{align}
\label{ttt}
\overline P_\varphi = \sum_{I \in \calB(\varphi)} \frac {(n-|I|)!} {(r-|I|)!} x_0^{r-|I|} \left ( \prod_{i \in I} x_i \right).
\end{align}
We first prove the next claim.
\medskip

\noindent
\underline{\textbf{Claim 1.}}
\begin{itemize}
\item[(I)] $a_0 \ne 0$.
\item[(II)] If $E \subset [n]$ is a $\varphi$-parallel class, then $\sum_{i \in E} a_i=0$.
\item[(III)]
For any flat $F$ of $M$  such that $\rank(\varphi(F))=\rank(F)$, one has
$\sum_{[n] \setminus F} a_i=-(n-r')a_0$.
\item[(IV)]
For any $j \in [n] \setminus L_\varphi$, we have $a_j=0$.
\end{itemize}

\begin{proof}[Proof of Claim]
(I) This follows from Theorem \ref{thm:LD}(i)
since $\overline P_\varphi= (n-r)! f_M + x_0 g$ for some polynomial $g \ne 0$.

(II)
Recall that $M'=([n],\calB({\varphi})_{r'})$ is a matroid on $[n]$.
Clearly $X \in \calB(\varphi)_{r'}$ if and only if $\rank(\varphi(X))=r'$ for any $X \subset [n]$. From this fact, it is easy to see that, for any $X \subset [n]$, the rank of $X$ in $M'$ equals to the rank of $\varphi(X)$ in $N$. In particular,
$i,j \in [n]$ are parallel in $M'$ if and only if they are $\varphi$-parallel.
Since $P_\varphi$ can be written in the form $P_\varphi=x_0^{n-r'} f_{M'} + h$, where $h$ is a polynomial that contains no monomial divisible by $x_0^{n-r'}$,
$\ell \overline P_\varphi$ can be written as
$$\ell \overline P_\varphi= \frac {(n-r')!}{(r-r')!} x_0^{r-r'} (a_1 \partial_1+ \cdots+a_n \partial_n) f_{M'} + h'$$
for some polynomial $h'$ containing no monomials divisible by $x_0^{r-r'}$.
Since $\ell \overline P_\varphi =0$, we have $(a_1 \partial_1+ \cdots+a_n \partial_n) f_{M'}=0$.
Then by Lemma \ref{lem:basicproperty}(iii) and Theorem \ref{thm:LD}(i) it follows that
$a_1 \partial_1+ \cdots + a_n \partial_n$ belongs to
$$
\mathrm{span}_\R \big\{ \{ \partial_i\mid \mbox{$i$ is an $\varphi$-loop}\} \cup \{\partial_i- \partial_j\mid \mbox{$i$ and $j$ are $\varphi$-parallel}\} \big\}.$$
This guarantees the desired property.

(III) The proof is similar to that of Theorem \ref{thm:LD}(i).
Suppose that $F$ has rank $r'$.
Let $I$ be an independent set of $M$ such that $\langle I \rangle =F$.
Note that $|I|= \rank(F)=r'$.
A routine computation tells that
the coefficient of $x_0^{r-r'-1} \prod_{i \in I} x_i$ in $\ell \overline P_\varphi$ is
\begin{align*}
&\frac {(n-r')!} {(r-r')!} (r-r')a_0
+\frac {(n-r'-1)!} {(r-r'-1)!} \left( \sum_{\{ j \} \cup I \in \calI(M),\ j \not \in I} a_j \right)\\
&=
\frac {(n-r'-1)!} {(r-r'-1)!} \left\{
(n-r')a_0+ \sum_{j \in [n] \setminus F} a_j \right\}
\end{align*}
(see also \eqref{ttt}). Since $\ell \overline P_\varphi=0$,
this proves the desired equation for $F$.

Now suppose that $F$ has rank $<r'$ and (III) holds for all flats $G \supsetneq F$ of $M$
with $\mathrm{rank} (\varphi(G))=\rank(G)$.
If $G$ is a smallest flat of $M$ that properly contains $F$, then
$$ \rank(\varphi(G)) -\rank(\varphi(F)) \leq \rank G - \rank F =1,$$
so the rank of $\varphi(G)$ must be either $\rank(F)+1$ or $\rank (F)$.
Let $G_1,\dots,G_p,G_1',\dots,G_q'$ be the minimal flats of $M$ that property contain $F$,
where $\mathrm{rank}(\varphi(G_k))= \mathrm{rank}(F)+1$ and $\mathrm{rank}(\varphi(G'_k))= \mathrm{rank} (F)$.
By Lemma \ref{lem:flat}, we have
$$[n] \setminus F= \bigsqcup_{t=1}^p (G_t \setminus F) \sqcup \bigsqcup_{s=1}^q (G_s'\setminus F).$$
We claim

\medskip

\noindent
\underline{\textbf{Claim 2.}}
$\bigsqcup_{t=1}^p (G_t \setminus F)$ is non-empty and a union of $\varphi$-parallel classes.

\begin{proof}[Proof of Claim 2]
Note that the definition of $G_1,\dots,G_p$ says that
$k \in \bigsqcup_{t=1}^p (G_t\setminus F)$ if and only if $\mathrm{rank}( \varphi(\{k\} \cup F))=\rank(F)+1$.
This in particular tells that $\bigsqcup_{t=1}^p (G_t \setminus F)$ is non-empty and contains no $\varphi$-loops.
If $i$ and $j$ are $\varphi$-parallel and $i \in \bigsqcup_{t=1}^p (G_t\setminus F)$ then we have
$$\mathrm{rank}( \varphi(\{j\} \cup F))=\rank (\varphi( \{i\} \cup F))= \rank (\varphi(F))+1,$$
which tells that $j \in \sqcup_{t=1}^p (G_t\setminus F)$.
This guarantees the desired property.
\end{proof}

Now, by statement (II), we have $\sum_{j \in \bigsqcup_{t=1}^p (G_t \setminus F)} a_j=0$.
Then
\begin{align*}
p \cdot \left( \sum_{j \in [n] \setminus F} a_j \right) & = \sum_{k=1}^p \left\{ \sum_{j \in [n] \setminus G_k} a_j + \sum_{j \in G_k \setminus F} a_j \right\}\\
&= \sum_{k=1}^p \left( \sum_{j \in [n] \setminus G_k} a_j \right)\\
&= p \times (n-r')a_0,
\end{align*}
which proves the desired property,
where we use the induction hypothesis to the third equality.

(IV)
If $|[n]\setminus L_\varphi|\leq 1$, then the assertion follows from the statement (II).
We assume  $|[n]\setminus L_\varphi|\geq 2$.
Let $\alpha=(n-r')a_0+\sum_{i \in L_\varphi} a_j$.
The statement (III) for rank $1$ flats tells that for any $k \in [n] \setminus L_\varphi$, we have
$\sum_{j \ne k} a_j =-(n-r')a_0$, equivalently,
$\sum_{j \in [n] \setminus L_\varphi,\ j \ne k} a_j=-\alpha$.
Then Lemma \ref{lem:linearalgebra} tells $a_j=-\frac {1} {|[n] \setminus L_\varphi|-1} \alpha$ for all $j \in [n] \setminus L_\varphi.$
Moreover, (II) tells, for any $j \in [n] \setminus L_\varphi$,
we have $0=\sum_{i \mbox{ \tiny is $\varphi$-parallel to $j$}} a_i=c \alpha$ for some $c<0$, so $\alpha=0$.
These prove the desired statement.
\end{proof}

We now go back to the proof of Theorem \ref{thm:morphism}.
By Claim 1, we have
$$ \textstyle \ell=a_0\partial_0+ \sum_{i \in L_\varphi} a_i \partial_i.$$
For each $I \in \calB(\varphi)$ with $I \subset [n] \setminus L_\varphi$,
let
$$N_I=\{ J \subset L_\varphi \mid J \cup I \in \calB(\varphi)\}.$$
Note that $N_I$ is the set of independent sets of the simple matroid obtained from $M$ by contracting elements in $I$ and then restrict it to $L_\varphi$.
Also,
$$P_\varphi= \sum_{I \in \calB(\varphi),\ I \subset [n] \setminus L_\varphi} x_0^{n-|I|-|L_\varphi|} \cdot P_{N_I} \cdot \left(\prod_{i \in I} x_i \right).$$
Then, since
$$\ell \overline P_\varphi=\ell \partial_0^{n-r} P_\varphi= \sum_{I \in \calB(\varphi),\ I \subset [n] \setminus L_\varphi} \left\{ \ell \partial_0^{n-r} \cdot \left( x_0^{n-|I|-|L_\varphi|} P_{N_I}\right) \right\}\left(\prod_{i \in I} x_i \right),$$
we have
$$\ell \partial_0^{n-r} \left( x_0^{n-|I|-|L_\varphi|} P_{N_I}\right)=0$$
for all $I \in \calB(\varphi)$ with $I \subset [n] \setminus L_\varphi$.
Also, by Lemma \ref{lem:4.1}, $N_I=M|_{L_\varphi}$ for all $I \in \calB(\varphi)$ with $|I|=r'$.
Then by Theorem \ref{thm:LD}(ii) and Lemma \ref{lem:QQQ}, we have either
$$ (\clubsuit)\ \mathrm{rank}(M|_{L_\varphi})\leq 1
\ \ \mbox{ or } \ \ (\spadesuit)\ \mbox{ $M|_{L_\varphi}$ is a uniform matroid and $n-r'-|L_\varphi|=0$}.$$

The latter case is nothing but the condition (C).
Suppose $\mathrm{rank}(M|_{L_\varphi})\leq 1$.
Then $L_\varphi = \emptyset$ or $|L_{\varphi}|=1$.
The former case cannot occur since $L_\varphi=\emptyset$ implies $\ell=a_0 \partial_0$ and the assumption $r>r'$ tells that $\overline P_\varphi$ contains a monomial divisible by $x_0$.
Suppose $L_\varphi=\{j_0\}$ for some $j_0 \in [n]$.
Then $\ell=a_0\partial_0+a_{j_0} \partial_{j_0}$.
Since $a_0 \ne 0$ and $a_0\partial_0 \overline P_\varphi=-a_{j_0}\partial_{j_0} \overline P_\varphi + \ell \overline P_\varphi=-a_{j_0}\partial_{j_0} \overline P_\varphi$,
we have
$$
a_0^2 \partial_0^2 \overline P_\varphi=
a_{j_0}^2 \partial_{j_0}^2 \overline P_\varphi=0.
$$
(Recall that $\overline P_\varphi$ contains no monomials which is divisible by $x_k^2$ for any $k \in [n]$.)
This tells that $\overline P_\varphi$ contains no monomial which is divisible by $x_0^2$.
This happens only when $r-r' \leq 1$.
Hence we have $|L_\varphi|=1$ and $r-r'=1$, so condition (B) is satisfied.
%
\end{proof}


Using the Lorentzian property of $P_\varphi$,
Eur--Huh \cite{EH} proved 
\[\fr<|\calB(\vphi)_{k-1}|/\binom{n}{k-1}>\fr<|\calB(\vphi)_{k+1}|/\binom{n}{k+1}>\leq \left(\fr<|\calB(\vphi)_{k}|/\binom{n}{k}>\right)^2 \ \ \ (r'<k<r).\]
Considering Corollary \ref{cor:main}, it is natural to ask

\begin{question}\label{q:question}
When equality holds in the above inequality? 
\end{question}

In the proof of Corollary \ref{cor:main},
we use the property that $\overline P_M$ has the $\HRR_1$ w.r.t.\ $\partial_1+ \cdots+\partial_n$.
We close this paper with an example showing that this is not the case for morphisms of matroids.

\begin{example}
Let $\vphi : U_{3, 3}\to U_{1,1}$ be a (unique) natural morphism. Then, 
\[\overline{P}_{\vphi}=x_1x_2x_3+x_0(x_1x_2+x_1x_3+x_2x_3)+x_0^2(x_1+x_2+x_3)\]
and a routine computation tells that $\overline{P}_{\vphi}$ does not have the $\SLP_1$ w.r.t.\ $\partial_1+\partial_2+ \partial_3$.
\end{example}



\end{document}